\documentclass[11pt]{article}
\usepackage{amsmath}
\usepackage{amssymb}
\usepackage{amsthm}
\usepackage{enumerate}
\usepackage{fancyhdr}
\usepackage[margin=4cm]{geometry}
\usepackage{tikz}
\usetikzlibrary{arrows}

\newcommand{\Q}{\ensuremath{\mathbb{Q}}}

\newcommand{\N}{\ensuremath{\mathbb{N}}}
\newcommand{\C}{\ensuremath{\mathbb{C}}}
\renewcommand{\P}{\ensuremath{\mathbb{P}}}
\newcommand{\A}{\ensuremath{\mathcal{A}}}

\newcommand{\V}{\ensuremath{\mathcal{V}}}
\newcommand{\E}{\ensuremath{\mathcal{E}}}

\DeclareMathOperator{\Ima}{Im}
\DeclareMathOperator{\rk}{rk}				
				
\DeclareMathOperator{\In}{In}				
\DeclareMathOperator{\nbc}{nbc}
\DeclareMathOperator{\Ext}{Ext}
\DeclareMathOperator{\Hom}{Hom}
\DeclareMathOperator{\gr}{gr}

\newtheorem{lemma}{Lemma}[subsection]
\newtheorem{theorem}[lemma]{Theorem}
\newtheorem{corollary}[lemma]{Corollary}
\newtheorem{proposition}[lemma]{Proposition}
\theoremstyle{definition}

\newtheorem{example}[lemma]{Example}
\newtheorem{remark}[lemma]{Remark}

\begin{document}

\title{Quadratic-linear duality and rational homotopy theory of chordal arrangements}
\author{Christin Bibby, Justin Hilburn}
\date{}
\maketitle

\begin{abstract}
To any graph and smooth algebraic curve $C$ one may associate a ``hypercurve'' arrangement and one can study the rational homotopy theory of the complement $X$. In the rational case ($C=\C$), there is considerable literature on the rational homotopy theory of $X$, and the trigonometric case ($C = \C^\times$) is similar in flavor. The case of when $C$ is a smooth projective curve of positive genus is more complicated due to the lack of formality of the complement. When the graph is chordal, we use quadratic-linear duality to compute the Malcev Lie algebra and the minimal model of $X$, and we prove that $X$ is rationally $K(\pi,1)$.
\end{abstract}

\section{Introduction}\label{intro}
This paper explores the topology of the rational, trigonometric, and projective (in particular elliptic) analogues of hyperplane arrangements. The rational case consists of linear arrangements, which are finite sets of codimension one linear subspaces of a complex vector space. The trigonometric case consists of toric arrangements, which are finite sets of codimension one subtori in a complex torus. The elliptic case consists of abelian arrangements, which are finite sets of codimension one abelian subvarieties of a product of elliptic curves.  We focus our attention on unimodular and supersolvable arrangements, which are classified by chordal graphs and are therefore called chordal arrangements. Chordal arrangements can be defined without reference to abelian group structure and hence make sense for curves of arbitrary genus.
When we discuss the projective case, we will only consider curves of \textit{positive} genus, as our method does not apply to $\P^1$ (whose cohomology ring is not Koszul). 

In each case, we study a differential graded algebra (DGA) that is a model (in the sense of rational homotopy theory)
for the complement of the arrangement.  
\begin{itemize}
\item For linear arrangements, the complement is formal, which means that the cohomology algebra with
trivial differential is itself a model.  A combinatorial presentation for this algebra is given by Orlik and Solomon \cite{orliksolomon}.
\item For toric arrangements, the complement is also formal, and in the unimodular case a combinatorial presentation of the cohomology
ring is given by De Concini and Procesi \cite{deconciniprocesi}.
\item In the projective case, the complement is not necessarily formal, but combinatorially presented models (with nontrivial differential) are given by the first author \cite{bibby} in the elliptic case and by Dupont \cite{dupont} in general.
\end{itemize}

When the matroid associated to the arrangement is supersolvable, the above model is Koszul; this is due to Shelton and Yuzvinsky for linear arrangements \cite{sheltonyuz},
and we prove it in the toric and projective cases (Theorems \ref{koszuldgathmtoric}, \ref{koszuldgathmabelian}, and \ref{koszuldgathmhighergenus}).
By studying the quadratic dual of the model, one can obtain a combinatorial presentation for a Lie algebra 
and use it to compute the $\Q$-nilpotent completion of the fundamental group and the minimal model.
This is done by Papadima and Yuzvinsky in the linear case \cite{papadimayuz}, and the toric case is completely analogous.
In the projective case, the lack of formality makes this computation more subtle: we need to use nonhomogeneous quadratic duality, 
where the dual to a Koszul differential graded algebra is a quadratic-linear algebra.
With this tool, we extend Papadima and Yuzvinsky's results to the projective setting (Theorem \ref{dualitythm}).

We also prove that complements of chordal arrangements are rational $K(\pi,1)$ spaces.
In the rational and toric cases, this follows from formality and Koszulity \cite{papadimayuz}.  In the projective case
(where we lack formality) it is not automatic, but we obtain it from our concrete description of the minimal model
(Corollary \ref{kpo}).

In the projective case, our results were inspired by \cite{romanbez} where Bezrukavnikov constructed a model of the ordered configuration space of an arbitrary smooth, projective, complex curve of positive genus;
he showed that his model was Koszul, gave a presentation for the dual Lie algebra, and described the minimal model. In fact our results generalize his since the ordered configuration space is the complement of the braid arrangement (which is chordal.)

In Section 2, we review known results on the cohomology of arrangements in each of our cases, giving explicit presentations for the algebras we will consider. 
In Section 3, we review the proof that the cohomology ring of the complement to a chordal linear arrangement is Koszul, and then we prove the analogous results in the toric and projective cases.
In Section 4, we review definitions and results from rational homotopy theory and quadratic-linear duality. 
The reader can skip ahead to Section 5 and refer back to Section 4 as needed. 
In Section 5, we review some known results on the rational homotopy theory of linear arrangements, which also apply to the toric case, and then we prove the analogous results for the projective case. 

\bigskip

\textbf{Acknowledgements. }
The authors would like to thank their advisor Nick Proudfoot for his many helpful comments, suggestions, and guidance. 
The first author would also like to thank Tyler Kloefkorn for helpful discussions on Koszul algebras.
Finally, the authors would like to thank an anonymous referee for suggesting a generalization to higher genus curves after the first draft of this paper. 

\section{Cohomology}\label{prelimsection}

In this section, we collect known results about the cohomology of the complement of a chordal arrangement in each of our cases. Since we will only be considering graphic arrangements in Sections \ref{koszulitysection} and \ref{topologysection} (see Remark \ref{chordalremark}), we will state all of the results in graphical language here, as well. Since our goal will be to study the rational homotopy theory of these spaces, we will also restrict our attention to cohomology with rational coefficients throughout this paper. 

\subsection{Definitions}

An \textbf{ordered graph} is a graph $\Gamma=(\V,\E)$ with an ordering on the vertices $\V$. 
We will assume throughout that our graphs are simple (that is, they have no loops or multiple edges).
An ordered graph can be considered as a directed graph in the following way:
For each edge $e\in\E$, label its larger vertex by $h(e)$ (for ``head'' of an arrow) and its smaller vertex by $t(e)$ (for ``tail'' of an arrow). 
An order on the vertices of $\Gamma$ induces an order on the edges by setting $e<e'$ if $h(e)<h(e')$ or if $h(e)=h(e')$ and $t(e)<t(e')$. 

\begin{remark}
None of the structures in this section will depend on the ordering of the vertices, but it will simplify the notation. The order chosen will also be necessary for the proofs in Section \ref{koszulitysection}.
\end{remark}

Let $\Gamma=(\V,\E)$ be an ordered graph. 
Let $C$ be $\C$, $\C^\times$, or a complex projective curve, and 
let $C^{\V}$ be the complex vector space (respectively torus or projective variety) whose coordinates are indexed by the vertices $\V$.
For each edge $e\in\E$, let $H_e=\{x_\V\in C^\V \ |\ x_{h(e)}=x_{t(e)}\}$. 
The collection $\A(\Gamma,C) = \{H_e\ |\ e\in\E\}$ is a \textbf{graphic} arrangement in $C^\V$. 
In each case, denote the complement of an arrangement $\A$ in $V$ by $X_\A:=V\setminus\cup_{H\in\A}H$.
In the case that $C$ is $\C$, $\C^\times$, or a complex elliptic curve, $\A(\Gamma,C)$ is a \textbf{linear}, \textbf{toric}, or \textbf{abelian} arrangement, respectively. 

\begin{example}
Let $C=\C$, $\C^\times$, or a complex projective curve. 
If $\Gamma$ is the complete graph on $n$ vertices, then $\A=\A(\Gamma,C)$ is the braid arrangement, and 
its complement $X_\A$ is the ordered configuration space of $n$ points on $C$.
\end{example}

\subsection{Linear Arrangements}

For linear arrangements, a combinatorial presentation for the cohomology ring was first given by Orlik and Solomon \cite{orliksolomon}. 
The fact that the complement of a linear arrangement is formal (that is, its cohomology ring is a model for the space; see Subsection \ref{ratlhomotopysubsection}) is originally due to Brieskorn \cite{brieskorn}. 
Here, we state these results for graphic arrangements. 

\begin{theorem}
\label{orliksolomonalgebra}
\emph{\cite[Theorems~3.126\&5.89]{orlikterao}}
Let $\Gamma=(\V,\E)$ be an ordered graph, and let $\A=\A(\Gamma,\C)$. 
Then $X_\A$ is formal and 
$H^*(X_\A,\Q)$ is isomorphic to the exterior algebra on the $\Q$-vector space spanned by 
$$\{g_e\ |\ e\in\E\}$$ modulo the ideal generated by 
\begin{enumerate}
\item[] $\sum_j(-1)^jg_{e_1}\cdots \hat{g}_{e_j}\cdots g_{e_k}$ whenever $\{e_1<\dots<e_k\}$ is a cycle.
\end{enumerate}
\end{theorem}

\subsection{Toric Arrangements}

De Concini and Procesi studied the cohomology of the complement of a toric arrangement.
If $\A$ is a unimodular toric arrangement (that is, all multiple intersections of subtori in $\A$ are connected), they show that the complement $X_\A$ is formal and give a presentation for the cohomology ring. Here, we state the result for graphic arrangements (which are always unimodular). 

\begin{theorem}
\label{toriccohomology}
\emph{\cite[Theorem~5.2]{deconciniprocesi}}
Let $\Gamma=(\V,\E)$ be an ordered graph, and let $\A=\A(\Gamma,\C^\times)$.
Then $X_\A$ is formal and 
$H^*(X_\A,\Q)$ is isomorphic to the exterior algebra on the $\Q$-vector space spanned by $$\{x_v,g_e\ |\ v\in\V, e\in\E\}$$ modulo the ideal generated by:
\begin{enumerate}[(i)]
\item[(ia)]
whenever $e_0,e_1,\dots,e_m$ is a cycle with $t(e_0)=t(e_1)$, $h(e_0)=h(e_m)$, and $h(e_i)=t(e_{i+1})$ for $i=1,\dots,m-1$ (as pictured below)

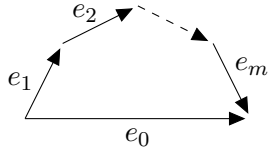
\begin{center}
\begin{tikzpicture}
[shorten >=2pt,>=triangle 45]
\draw[->](.5,0)-- node[left]{$e_1$} (1,1);
\draw[->](1,1)-- node[above,pos=.3]{$e_2$} (2,1.5);
\draw[->,dashed](2,1.5)--(3,1);
\draw[->](3,1)-- node[right,pos=.3]{$e_m$} (3.5,0);
\draw[->](.5,0)-- node[below]{$e_0$} (3.5,0);
\end{tikzpicture}
\end{center}

we have
$$g_{e_1}g_{e_2}\cdots g_{e_m} - \sum (-1)^{|I|+m+s_I}g_{e_{i_1}}\cdots g_{e_{i_k}}\psi_{e_{j_1}}\cdots \psi_{e_{j_{m-k-1}}}g_{e_0}$$ where the sum is taken over all $I=\{i_1<\cdots<i_k\}\subsetneq\{1,\dots,m\}$ with complement $\{j_1<\cdots<j_{m-k}\}$, $\psi_{e_\ell}=x_{h(e_\ell)}-x_{t(e_{\ell})}$, and $s_I$ is the parity of the permutation $(i_1,\dots,i_k,j_1,\dots,j_{m-k})$. 
\item[(ib)] if we again have a cycle, but have some arrows reversed, relabel the arrows so that $e_1<\dots<e_s<e_0$, 
then take the relation from (ia) and replace each $\psi_{e_i}$ with $-\psi_{e_i}$ and each $g_{e_i}$ with $-g_{e_i}-\psi_{e_i}$ whenever $e_i$ points in the opposite direction of $e_0$. 
\item[(ii)] $(x_{h(e)}-x_{t(e)})g_e$ for $e\in\E$.
\end{enumerate}
\end{theorem}

The presentation encodes both the combinatorics of the arrangement and the geometry of the ambient space. 
The generators $x_v$ come from the cohomology of $\C^\times$, while the generators $g_e$ are similar to that of the Orlik-Solomon algebra for its rational counterpart. 
However, the toric analogue of the Orlik-Solomon relation is much more complicated.

\subsection{Abelian Arrangements}

The elliptic analogue has a very different flavor, since the complement to an arrangement is not formal. 
If $\A$ is a unimodular abelian arrangement (that is, all multiple intersections of subvarieties in $\A$ are connected), the first author gave a presentation for a model for $X_\A$ \cite[Thorem~4.1]{bibby}. 
The presentation for graphic abelian arrangements is also a special case of one given by Dupont and Bloch \cite{dupont}, which we state in the next subsection. 

\begin{theorem}
\label{modelthm}
Let $E$ be a complex elliptic curve. 
Let $\Gamma=(\V,\E)$ be an ordered graph, and let $\A=\A(\Gamma,E)$. 
Define the differential graded algebra $A(\A)$ as the exterior algebra on the $\Q$-vector space spanned by $$\{x_v,y_v,g_e\ |\ v\in\V,e\in\E\}$$ modulo the ideal generated by the following relations:
\begin{enumerate}[(i)]
\item $\sum_j(-1)^jg_{e_1}\cdots \hat{g}_{e_j}\cdots g_{e_k}$ whenever $\{e_1<\dots<e_k\}$ is a cycle and
\item $(x_{h(e)}-x_{t(e)})g_e$, $(y_{h(e)}-y_{t(e)})g_e$, for each $e\in\E$.
\end{enumerate}
The differential is defined by putting $dx_v=dy_v=0$ and $$dg_e=(x_{h(e)}-x_{t(e)})(y_{h(e)}-y_{t(e)}).$$ The DGA $(A(\A),d)$ is a model for $X_\A$. 
\end{theorem}

In a similar way to toric arrangements, this algebra encodes both the combinatorics of the arrangement (with the Orlik-Solomon relation) and the geometry of the ambient space.
The generators $x_v,y_v$ come from the cohomology of $E$, while the $g_e$ are from the Orlik-Solomon algebra of its rational counterpart. 

\subsection{Higher Genus Curves}

By the work of Dupont and Bloch, we have the following presentation for graphic arrangements in the case that $C$ is a complex projective curve of positive genus \cite{dupont}, which we state here. 

\begin{theorem}
\label{modelthm2}
Let $C$ be a complex projective curve with genus $g\geq 1$. 
Define the differential graded algebra $A(\A)$ as the exterior algebra on the $\Q$-vector space spanned by $$\{x^i_v,y^i_v,g_e\ |\ v\in\V,e\in\E,i=1,\dots,g\}$$ modulo the ideal generated by the following relations:
\begin{enumerate}[(i)]
\item $\sum_j(-1)^jg_{e_1}\cdots \hat{g}_{e_j}\cdots g_{e_k}$ whenever $\{e_1<\dots<e_k\}$ is a cycle,
\item $(x^i_{h(e)}-x^i_{t(e)})g_e$, $(y^i_{h(e)}-y^i_{t(e)})g_e$, for each $e\in\E$,
\item[(iiia)] $x_v^iy_v^j$, $x_v^ix_v^j$, $y_v^iy_v^j$ for $i\neq j$, and 
\item[(iiib)] $x_v^iy_v^i-x_v^jy_v^j$. 
\end{enumerate}
The differential is defined by putting $dx^i_v=dy^i_v=0$ and $$dg_e=x^1_{h(e)}y^1_{h(e)}+x^1_{t(e)}y^1_{t(e)} - \sum_{i=1}^g \left(x^i_{h(e)}y^i_{t(e)} + x^i_{t(e)}y^i_{h(e)}\right).$$
The DGA $(A(\A),d)$ is a model for $X_\A$.
\end{theorem}

Just as before, this algebra encodes both the combinatorics of the arrangement and the geometry of the ambient space. 
The generators $x^i_v,y^i_v$ come from the cohomology of $C^\V$, and we write these generators and relations here explicitly since we will use this presentation  to show that the algebra is Koszul in Section \ref{highergenuskoszulity}.
A more elegant way of writing the differential is to say that the generator $g_e$ maps to $[\Delta_e]\in H^2(C^\V)$, where $\Delta_e$ is the diagonal corresponding to the coordinates indexed by $h(e)$ and $t(e)$ in $C^\V$. 

\section{Koszulity}\label{koszulitysection}

In this section, we will show that for chordal arrangements, the algebras presented in Theorems \ref{orliksolomonalgebra}, \ref{toriccohomology}, \ref{modelthm}, and \ref{modelthm2} are Koszul.
The cohomology of the complement of a chordal linear arrangement was first shown to be Koszul by Shelton and Yuzvinsky \cite{sheltonyuz}. In Subsection \ref{linearkoszulity}, we outline the proof presented by Yuzvinsky in \cite{yuz}. 
The analogous results for toric and abelian arrangements, as well as for higher genus curves, are new and presented in Subsections \ref{torickoszulity}, \ref{abeliankoszulity}, and \ref{highergenuskoszulity}, respectively. 

\subsection{Chordal Arrangements}

Let $C$ be $\C$, $\C^\times$, or a complex elliptic curve, and let $\Gamma=(\V,\E)$ be a simple graph. If $\Gamma$ is chordal (that is, every cycle with more than three vertices has a chord), then the graphic arrangement $\A(\Gamma,C)$ is said to be \textbf{chordal}. 

A \textbf{perfect elimination ordering} is an order on the vertices so that for all $v\in\V$, $v$ is a simplicial vertex (a vertex whose neighbors form a clique) in the graph $\Gamma_v:=\Gamma-\{v'\in\V\ |\ v'>v\}$. 
Such an ordering exists if and only if $\Gamma$ is chordal \cite[p.~851]{fulkersongross}.
From now on, we will use such an order when discussing chordal graphs.

We say a set $S=\{e_1<\cdots<e_k\}$ is a \textbf{broken circuit} if there is some edge $e$ with $e<e_1$ such that $S\cup\{e\}$ is a cycle. A set $S\subseteq\E$ is \textbf{nbc} (non-broken circuit) if no subset of it is a broken circuit. 
Let $F\subseteq\E$ be a flat of the matroid of $\Gamma$, and consider the subgraph $\Gamma[F]$ of $\Gamma$, which has edges $F$ and vertices adjacent to edges in $F$. 
We say that an $\nbc$ set $S$ is associated to $F$ if $S\subseteq F$ and $S$ spans $\Gamma[F]$.

\begin{remark}
\label{chordalremark}
In the case of linear, toric, or abelian arrangements, 
the essential property that we need for our results is that the arrangement is unimodular (for Theorems \ref{toriccohomology} and \ref{modelthm}) and supersolvable (for Theorems \ref{koszuldgathmtoric} and \ref{koszuldgathmabelian}). 
We could state all of our results in the language of unimodular and supersolvable arrangements; however, this isn't any more general than the language of chordal graphs.
This is because Ziegler showed that a matroid is unimodular and supersolvable if and only if it is chordal graphic (Proposition 2.6 and Theorem 2.7 of \cite{ziegler}). 
In fact, since the edge set of $\Gamma-v$ is a modular hyperplane when $v$ is a simplicial vertex \cite[Proposition~4.4]{ziegler}, the maximal chain of modular flats in the matroid corresponds exactly to our ordering on the vertices.
\end{remark}

\subsection{Linear Arrangements}\label{linearkoszulity}

Yuzvinsky proved that the Orlik-Solomon ideal has a quadratic Gr\"obner basis when $\A$ is supersolvable (eg. chordal), which implies that $H^*(X_\A)$ is Koszul \cite[Corollary~6.21]{yuz}. 
We outline his technique as we will use similar techniques in the toric and abelian cases. 

For ease of notation, whenever $C=\{e_1<\cdots<e_k\}$ we will use $g_C:=g_{e_1}\cdots g_{e_k}$ and $\partial g_C:=\sum_j(-1)^jg_{e_1}\cdots\hat{g}_{e_j}\cdots g_{e_k}$.
Let $\Gamma=(\V,\E)$ be a chordal graph with a perfect elimination ordering on the vertices (and edges). 

First, Yuzvinsky showed that the set $G=\{\partial g_C\ |\ C \text{ is a circuit}\}$ is a Gr\"obner basis for the ideal $I=\langle G\rangle$ in the exterior algebra $\Lambda(g_e\ |\ e\in\E)$, with the degree-lexicographic order such that $g_e<g_{e'}$ whenever $e<e'$. 
The leading (or initial) term of $\partial g_C$ is $\In(\partial g_C)=g_{C'}$ where $C'\subseteq C$ is the broken circuit associated to $C$.
Recall that a subset $G$ of an ideal $I$ is a Gr\"obner basis if $\In(I)=\langle \In(G)\rangle$.
To prove that this is a Gr\"obner basis, Yuzvinsky used the fact that the set of monomials not in $\In(I)$ form a basis for $H^*(X_\A)=\Lambda(g_e)/I$. 
The set of monomials not in $\In(I)$ is the basis $\{g_C\ |\ C\text{ is }\nbc\}$. 

Moreover, since $\Gamma$ is chordal, this Gr\"obner basis can be reduced to a quadratic Gr\"obner basis. This is because we have the following property (which follows immediately from Proposition 6.19 of \cite{yuz}): $S\subseteq\E$ is an $\nbc$ set if and only if for all distinct $e,e'\in S$ we have $h(e)\neq h(e')$. 
A circuit $C$ is not $\nbc$, hence there exist distinct edges $e,e'\in C$ such that $h(e)=h(e')$. But then $\{e,e'\}$ contains (and hence is) a broken circuit, and so it is contained in some circuit $T$ with $|T|=3$. Thus $\In(\partial g_T)=g_{e}g_{e'}$ divides $\In(\partial g_C)$, and we can reduce our Gr\"obner basis to a quadratic one.

\subsection{Toric Arrangements}\label{torickoszulity}

Since the complement to a chordal toric arrangement is formal (as in the linear case), we want to show that its cohomology ring is Koszul. 
Our argument will be similar to (but slightly more complicated than) the linear case.
We will provide a $\Q$-basis for the cohomology ring, use it to show that our generating set of the ideal is a Gr\"obner basis, and then reduce the Gr\"obner basis to a quadratic one.

\begin{lemma}
\label{choicelemmatoric}
Let $\Gamma=(\V,\E)$ be a chordal graph. 
Let $F$ be a flat of the arrangement $\A=\A(\Gamma,\C^\times)$, and let $S$ be a non-broken circuit associated to $F$.  
Define $I_F$ to be the ideal generated by 
$$\left\{x_{h(e)}-x_{t(e)}\ |\ e\in F\right\}$$ 
in $\Lambda(x_v\ |\ v\in\V)\cong H^*((\C^\times)^\V)$, and let $H_F=\cap_{e\in F}H_e\subseteq (\C^\times)^\V$. 
\begin{enumerate}
\item With the degree-lexicographic order and $x_v<x_{v'}$ whenever $v<v'$, the set 
$$G_S := \left\{ x_{h(e)}-x_{t(e)}\ |\ e\in S\right\}$$
is a Gr\"obner basis for $I_F$.
\item The set $\{x_{i_1}\cdots x_{i_r}\ |\ h(e)\notin \{i_1,\dots,i_r\} \text{ for each } e\in S\}$ is a basis for $$H^*(H_F) \cong \Lambda(x_v\ |\ v\in\V)/I_F$$ and this basis does not depend on the choice of $\nbc$ set $S$.
\end{enumerate}
\end{lemma}
\begin{proof}
For linear relations, finding a Gr\"obner basis is equivalent to Gaussian elimination, and so consider the matrix $M_F$  
whose rows are indexed by edges $e\in F$, whose columns are indexed by the vertices in decreasing order, and whose entries are zero except  $(M_F)_{e,h(e)}=1$ and $(M_F)_{e,t(e)}=-1$ (so that row $e$ corresponds to the element $x_{h(e)}-x_{t(e)}$).
Note that since  $|S|=\rk(S)=\rk(M_F)$, we may use row operations so that the rows corresponding to elements of $S$ remain unchanged while all other rows are zero. Moreover, since $|\{h(e)\ |\ H_e\in S\}|=|S|$, the matrix is in row echelon form.
Thus, $G_S$ is a Gr\"obner basis for $I_F$.

For part (2), by Gr\"obner basis theory, the set of monomials not in $\In(I_F)$ form a basis for $H^*(H_F)$. Since the ideal $\In(I_F)$ is generated by $\In(G_S) = \{x_{h(e)}|e\in S\}$, the monomials not in $\In(I_F)$ are precisely those stated.
Since $S$ is an $\nbc$ set associated to $F$, it spans the subgraph $\Gamma[F]$. 
Thus $\{h(e)\ |\ e\in F\}=\{h(e)\ |\ e\in S\}$ and the basis given does not depend on $S$.
\end{proof}

For ease of notation, we will use $$x_Ag_C:= x_{a_1}\cdots x_{a_r}g_{c_1}\cdots g_{c_k}$$ 
where $A=\{a_1<\dots<a_r\}$ and $C=\{c_1<\dots<c_k\}$.
We will also denote relations (ia) and (ib) from Theorem \ref{toriccohomology} by $r_C$ for a cycle $C$. 

\begin{lemma}
\label{basislemmatoric}
Let $\Gamma=(\V,\E)$ be a chordal graph, and let $\A=\A(\Gamma,\C^\times)$. 
Define $P$ to be the set of all monomials $x_Ag_C$ such that $C$ is a non-broken circuit and $h(e)\notin A$ for all $e\in C$. Then $P$ is a basis for $H^*(X_\A)$.
\end{lemma}
\begin{proof}
There is a decomposition into the flats of $\A$ \cite[Remark~4.3(2)]{deconciniprocesi} (see also \cite[Lemma~3.1]{bibby}), which is given by the following:
For a flat $F$, let $H_F=\cap_{e\in F}H_e\subseteq(\C^\times)^\V$, and let $V_F$ be the vector space spanned by $g_C$ for all $\nbc$ sets $C$ associated to $F$. 
Then $$H^*(X_\A)=\oplus_F H^*(H_F)\otimes V_F.$$
Denote $H^*(H_F)\otimes V_F$ by $A_F$. 
To show that $P$ is a basis for $H^*(X_\A)$, it suffices to show that 
$$P\cap A_F=\{x_Ag_C\ |\ h(c)\notin A\text{ for }c\in C, C\text{ is an nbc set associated to }F\}$$
is a basis for $A_F$.
But this follows from Lemma \ref{choicelemmatoric}.
\end{proof}

\begin{theorem} 
\label{koszuldgathmtoric}
Let $\A$ be a chordal toric arrangement. Then $H^*(X_\A)$ is  Koszul.
\end{theorem}
\begin{proof}
Fix a degree-lexicographic order on $H^*(X_\A)$ that is induced by our order on $\V$. That is, $g_e<g_{e'}$ if $e<e'$, and $x_{h(e)}<g_e<x_{h(e)+1}$.
We will show that $$G=\left\{(x_{h(e)}-x_{t(e)})g_e,r_S \ |\ e\in\E,\text{ $S$ a circuit}\right\}$$
is a Gr\"obner basis with this order which can be reduced to a quadratic Gr\"obner basis.

 We have 
$$\In(G) = \left\{x_{h(e)}g_e,\ g_C\ |\ e\in\E,C\text{ is a broken circuit}\right\}.$$ 
Then $P$ is the set of monomials that are not in $\langle \In(G)\rangle$. Since $\langle\In(G)\rangle\subseteq\In(I)$, $P$ contains the monomials that are not in $\In(I)$. Since the set of monomials not in $\In(I)$ is a basis for $H^*(X_\A)$ contained in the basis $P$, and $H^*(X_\A)$ is finite dimensional, we must have equality throughout. That means that the monomials in $\langle\In(G)\rangle$ are exactly the monomials in $\In(I)$. Since these ideals are generated by monomials, they must be equal.  
Note that the relations of type (ii) are already quadratic. In a similar way as in the linear case, we can reduce our relations $r_C$ to quadratic ones as well. 
\end{proof}

\subsection{Abelian Arrangements}\label{abeliankoszulity}

Let $\Gamma=(\V,\E)$ be a chordal graph, and let $E$ be a complex elliptic curve. 
For the chordal abelian arrangement $\A=\A(\Gamma,E)$, consider the algebra $A(\A)$ from Theorem \ref{modelthm} (ignoring the differential). 
In this subsection, we will prove that $A(\A)$ is Koszul. 
The proof is very similar to (but slightly more complicated than) the toric case. 

\begin{lemma}
\label{choicelemmaabelian}
Let $\Gamma=(\V,\E)$ be a chordal graph. 
Let $F\subseteq\E$ be a flat of the arrangement $\A=\A(\Gamma, E)$, and let $S$ be a non-broken circuit associated to $F$.  
Define $I_F$ to be the ideal generated by 
$$\left\{x_{h(e)}-x_{t(e)}, y_{h(e)}-y_{t(e)}\ |\ e\in F\right\}$$ 
in $\Lambda(x_v,y_v\ |\ v\in\V)\cong H^*(E^\V)$, and let $H_F=\cap_{e\in F}H_e\subseteq E^\V$.
\begin{enumerate}
\item With the degree-lexicographic order and $x_v<y_v<x_{v'}<y_{v'}$ whenever $v<v'$, the set 
$G_S := \left\{ x_{h(e)}-x_{t(e)}, y_{h(e)}-y_{t(e)}\ |\ e\in S\right\}$
is a Gr\"obner basis for $I_F$.
\item The set $\{x_{i_1}\cdots x_{i_r}y_{j_1}\cdots y_{j_t}\ |\ h(e)\notin \{i_1,\dots,i_r,j_1,\dots,j_t\} \text{ for each } e\in S\}$ is a basis for $$H^*(H_F) \cong \Lambda(x_v,y_v\ |\ v\in\V)/I_F$$ and this basis does not depend on the choice of $\nbc$ set $S$.
\end{enumerate}
\end{lemma}
\begin{proof}
Consider the matrix $M_F$ from the proof of Lemma \ref{choicelemmatoric}. Build a $2\times2$ block matrix, where the upper left and lower right blocks are copies of $M_F$ and the other blocks are zero. In the upper half of the matrix, row $e$ corresponds to $x_{h(e)}-x_{t(e)}$, and in the lower half of the matrix, row $e$ corresponds to $y_{h(e)}-y_{t(e)}$. By a similar argument as before, we can eliminate rows that don't correspond to elements of $S$ and we're left with a matrix in row echelon form.
Thus, we have a Gr\"obner basis.

The proof of the second statement mimics the proof in the toric case, with $\In(G_S)=\{x_{h(e)},y_{h(e)}\ |\ e\in S\}$.
\end{proof}

\begin{lemma}
\label{basislemmaabelian}
Let $\Gamma=(\V,\E)$ be a chordal graph, and let $\A=\A(\Gamma,E)$. 
Define $P$ to be the set of all monomials $x_Ay_Bg_C$ such that $C$ is a non-broken circuit and $h(e)\notin(A\cup B)$ for all $e\in C$. Then $P$ is a basis for $A(\A)$.
\end{lemma}
\begin{proof}
By Lemma 3.1 in \cite{bibby}, there is a decomposition into the flats of $\A$, given by the following:
For a flat $F$, let $H_F=\cap_{e\in F}H_e\subseteq E^\V$, and let $V_F$ be the vector space spanned by $g_C$ for all $\nbc$ sets $C$ associated to $F$. 
Then $$A(\A) = \oplus_F H^*(H_F)\otimes V_F.$$
Denote $H^*(H_F)\otimes V_F$ by $A_F$. 
To show that $P$ is a basis for $A(\A)$, it suffices to show that 
$$P\cap A_F=\{x_Ay_Bg_C\ |\ h(c)\notin (A\cup B)\text{ for }c\in C, C\text{ is an nbc set associated to }F\}$$
is a basis for $A_F$.
But this follows from Lemma \ref{choicelemmaabelian}.
\end{proof}

\begin{theorem} 
\label{koszuldgathmabelian}
Let $\A$ be a chordal abelian arrangement. Then $A(\A)$ is Koszul. 
\end{theorem}
\begin{proof}
Fix a degree-lexicographic order on $A(\A)$ that is induced by our order on $\V$. That is, $g_e<g_{e'}$ if $e<e'$, and $x_{h(e)}<y_{h(e)}<g_e<x_{h(e)+1}<y_{h(e)+1}$.
We claim that $$G=\left\{(x_{h(e)}-x_{t(e)})g_e,(y_{h(e)}-y_{t(e)})g_e,\partial g_S \ |\ e\in\E,\text{ $S$ a circuit}\right\}$$
is a Gr\"obner basis with this order.
Here, 
$$\In(G) = \left\{x_{h(e)}g_e,\ y_{h(e)}g_e,\ g_C\ |\ e\in\E,C\text{ is a broken circuit}\right\}$$ 
and $P$ from Lemma \ref{basislemmaabelian} is the set of monomials not in $\langle\In(G)\rangle$. By an argument similar to that in the toric case, we can conclude that $G$ is a Gr\"obner basis. Moreover, using the fact that we have a chordal graph, we can again reduce this (in the same way) to a quadratic Gr\"obner basis, thus proving Koszulity. 
\end{proof}

\subsection{Higher Genus Curves}\label{highergenuskoszulity}

Let $\Gamma=(\V,\E)$ be a chordal graph, and let $C$ be a complex projective curve of genus $g>1$. For the chordal arrangement $\A=\A(\Gamma,C)$, consider the algebra $A(\A)$ from Theorem \ref{modelthm2} (ignoring the differential).  In this subsection, we will prove that $A(\A)$ is Koszul. The proof is very similar to that of the abelian case in Section \ref{abeliankoszulity}. 

\begin{lemma}
\label{choicelemmahighergenus}
Let $\Gamma=(\V,\E)$ be a chordal graph. Let $F\subseteq\E$ be a flat of the arrangement $\A=\A(\Gamma,C)$, and let $S$ be a non-broken circuit associated to $F$. 
Denote $H_F = \cap_{e\in F}H_e \subseteq C^\V$. 
Then $H^*(H_F) \cong \Lambda(x_v^i,y_v^i\ |\ v\in\V, i=1,\dots,g)/I_F$ where $I_F$ is the ideal generated by the relations
\begin{enumerate}[(i)]
\item $x_{h(e)}^i-x_{t(e)}^i$, $y_{h(e)}^i-y_{t(e)}^i$ for $e\in F$,
\item $x_v^ix_v^j$, $y_v^iy_v^j$, $x_v^iy_v^j$ for $i\neq j$, and 
\item $x_v^iy_v^i-x_v^jy_v^j$.
\end{enumerate}
This algebra has basis $$\{x^1_{A_1}\cdots x^g_{A_g}y^1_{B_1}\cdots y^g_{B_g}\ |\ A_i\cap B_i=\emptyset\text{ for }i>1; \{h(e)\ |\ e\in S\}\cap(A_i\cup B_i)=\emptyset \forall i\}$$ and this basis does not depend on the choice of $S$. 
\end{lemma}
\begin{proof}
Consider the exterior algebra modulo the first relation, which we can write as the exterior algebra
$$\Lambda(x^i_v,y^i_v\ |\ v \notin\{h(e)\ |\ e\in S\},i=1,\dots,g)$$
by a similar argument as in the proof of Lemma \ref{choicelemmaabelian}. Note that as before $\{h(e)\ |\ e\in S\}=\{h(e)\ |\ e\in F\}$, and so this does not depend on the choice of $S$. Now consider relations (ii) and (iii) in this algebra. This is a Gr\"obner basis $G$ with 
$$\In(G) = \{
x_v^ix_v^j, y_v^iy_v^j, x_v^iy_v^j (i\neq j), x_v^iy_v^i (i>1), v\notin \{h(e)\ |\ e\in S\} \}$$
The set of monomials in our proposed basis are exactly those not divisible by $\In(G)$ and are hence a basis.
\end{proof}

\begin{lemma}
\label{basislemmahighergenus}
Let $\Gamma=(\V,\E)$ be a chordal graph, and let $\A=\A(\Gamma,C)$. Define $P$ to be the set of all monomials $x^1_{A_1}\cdots x^g_{A_g}y^1_{A_1}\cdots y^g_{A_g}g_S$ such that $S$ is a non-broken circuit, $h(e)\notin (A_i\cup B_i)$ for all $e\in S$ and all $i$, and $A_i\cap B_i=\emptyset$ for $i>1$. Then $P$ is a basis for $A(\A)$.
\end{lemma}
\begin{proof}
There is again a decomposition into the flats of $\A$, given by the following: For a flat $F$, let $H_F=\cap_{e\in F}H_e\subseteq C^\V$, and let $V_F$ be the vector space spanned by $g_S$ for all $\nbc$ sets $S$ associated to $F$. Then $$A(\A) = \oplus_FH^*(H_F)\otimes V_F.$$
Denote $H^*(H_F)\otimes V_F$ by $A_F$. To show that $P$ is a basis for $A(\A)$, it suffices to show that $P\cap A_F$ is a basis of $A_F$. But this follows from Lemma \ref{choicelemmahighergenus}. 
\end{proof}

\begin{theorem}
\label{koszuldgathmhighergenus}
Let $C$ be a complex projective curve of genus $g>1$, and let $A$ be a chordal arrangement in $C^\V$. Then $A(\A)$ is Koszul. 
\end{theorem}
\begin{proof}
Fix a degree-lexicographic order on $A(\A)$ that is induced by our order on $\V$. We claim that 
$$G = \{ (x^i_{h(e)}-x^i_{t(e)})g_e,(y^i_{h(e)}-y^i_{t(e)})g_e, \partial g_S, R\ |\ e\in\E, S \text{ a circuit} \}$$
is a Gr\"obner basis with this order, where $R$ denotes the set of relations (iiia) and (iiib) in $A(\A)$. Here,
$$\In(G) = \{x_v^1y_v^1, x_{h(e)}^ig_e, y_{h(e)}^ig_e, g_B\ |\ B \text{ a broken circuit}\} $$
and $P$ from Lemma \ref{basislemmahighergenus} is the set of monomials not in $\langle\In(G)\rangle$. By an argument similar to the previous cases, we can conclude that $G$ is a Gr\"obner basis. Moreover, using the fact that we have a chordal graph, we can again reduce this (in the same way) to a quadratic Gr\"obner basis, thus proving Koszulity. 
\end{proof}

\section{Rational Homotopy Theory and Quadratic Duality}\label{dualitysection}

In this section, we collect definitions and  results from rational homotopy theory, quadratic duality, and the relationship between these two subjects. 
This section is meant to provide a background on the necessary theory; the reader can skip ahead and refer back as needed. 
Throughout this section, all DGAs will be assumed to be connective (that is, their cohomology has a nonnegative grading). Except in Subsection \ref{dualitysubsection}, all DGAs will be commutative (that is, graded-commutative).

\subsection{Rational Homotopy Theory}
\label{ratlhomotopysubsection} 
The fundamental problem of rational homotopy theory is to understand the topology of the $\Q$-completion $X \to \Q_{\infty}(X)$ of a topological space $X$ as defined in \cite[Chapter I.4]{bousfieldkan}. When $X$ is a simply connected CW-complex, we have $\pi_i \Q_{\infty}(X) \cong (\pi_i X) \otimes \Q$ and $H^*(X, \Q) \cong H^*(\Q_{\infty}(X), \Q)$, but in general the relationship between $X$ and $\Q_{\infty}(X)$ is more complicated.
Still, the homotopy type of $\Q_{\infty}(X)$ is substantially simpler than that of $X$ as the results of \cite{quillen, sullivan, bousfield} show that the rational homotopy theory of connected $\Q$-finite spaces is determined by the quasi-isomorphism type of a particular DGA $(A_{PL}(X), d)$ with $H^*(A_{PL}(X), d) \cong H^*(X, \Q)$.
A DGA $(A(X),d)$ is a \textbf{model} for $X$ if it is quasi-isomorphic to $(A_{PL}(X),d)$. The space $X$ is \textbf{formal} if $(H^*(X,\Q),0)$ is a model for $X$. 

Let $(B,d)$ be a DGA and for $n \geq 0$ let $(B(n), d)$ be the DG-subalgebra of $B$ generated by $B^i$ for $i \leq n$. Define $(B(-1), d)$ to be the DG-subalgebra generated by $1 \in B^0$. For $n \geq 0$, there is an increasing filtration $(B(n,q), d)$ on $(B(n), d)$ defined inductively as follows: Let $(B(n,0), d) = (B(n-1), d)$ and let $B(n,q+1)$ be the DG-subalgebra of $B$ generated by $B(n-1)$ and $\{b \in B^n | db \in B(n, q) \}$. A commutative DGA $(B,d)$ is \textbf{minimal} \cite[Section~7.1]{bousfield} if it is connected, $B$ is a free commutative graded algebra, and $B(n) = \cup_{q \geq 0} B(n,q)$ for all $n$. 

Sullivan \cite{sullivan} showed that any homologically connected DGA has a minimal model that is unique up to unnatural isomorphism. Write $(M(X), d)$ for the minimal model of $(A_{PL}(X), d)$, which is called the minimal model of $X$. Every minimal DGA $(M, d)$ has a canonical augmentation determined by the augmentation ideal $M^{+}$. This lets us define the homotopy groups $\pi^q M = H^q(M^+/(M^+ \cdot M^+), d)$ of $(M, d)$. The following theorem relates the homotopy groups of $(M(X), d)$ to those of $\Q_{\infty}(X)$.
\begin{theorem} \emph{\cite[Theorem~12.8]{bousfield}} \label{ratlhomotopyprop}
There are natural bijections $$\pi_q \Q_{\infty}(X) \cong \Hom_{\Q}(\pi^q M(X) , \Q).$$ 
They are group isomorphisms for $q \geq 2$.
\end{theorem}

In the next subsection we develop the technology to get more refined information about $\pi_1 X$ from the minimal model. This will be important because the spaces we are interested in are rational $K(\pi,1)$ spaces. 

\subsection{Complete Lie Algebras and Nilpotent Completion of Groups}
When  $X$ is not simply connected, we don't necessarily have the isomorphism $\pi_1\Q_\infty(X)\cong(\pi_1 X)\otimes \Q$. To even make sense of the right hand side when $\pi_1X$ isn't abelian, 
we need to review the \textbf{Malcev completion} (or $\Q$-nilpotent completion) $\hat{G} \otimes \Q$ of a finitely presented group $G$. Then we will survey some results that show that $\hat{G} \otimes \Q$ is entirely determined by a complete Lie algebra $L(G)$, which we call the \textbf{Malcev Lie algebra} of $G$. 

Recall that the lower central series of $G$ is defined by setting $\Gamma_1 G = G$ and $\Gamma_{q+1} G = [G, \Gamma_q G]$. As in \cite[Section~13.2]{griffithsmorgan}, we will use a recursive procedure to define $N_i G\otimes\Q$ for each nilpotent group $N_i G = G / \Gamma_i G$ and then define $$\hat{G} \otimes \Q := \varprojlim (N_i G) \otimes \Q.$$
First we see that $N_1 G = 0$ so we can define $(N_1 G) \otimes \Q = 0$. Now assume that we have defined $(N_{i-1} G) \otimes \Q$. The $N_i G$ fit into a series of exact sequences $$0 \to \Gamma_{i-1} / \Gamma_i G \to N_i G \to N_{i-1} G \to 0$$ which determine classes $\sigma_i \in H^2(N_{i-1}G, \Gamma_{i-1} / \Gamma_i G)$ (where $\Gamma_{i-1} / \Gamma_i G$ is given a trivial $N_{i-1} G$-module structure). It can be shown that $$H^2(N_{i-1}, \Gamma_{i-1} / \Gamma_i G) \otimes \Q \cong H^2((N_{i-1} G) \otimes \Q, (\Gamma_{i-1} / \Gamma_i G) \otimes \Q)$$ so the class $\sigma_i \otimes 1$ determines an extension 
of $(N_{i-1}G)\otimes\Q$ by $(\Gamma_{i-1}/\Gamma_iG)\otimes\Q$. We then define this extension to be $N_iG\otimes \Q$. 
 
For a minimal DGA $(M, d)$ the analogue of the lower central series of $\pi_1 X$ is an increasing filtration on $\pi^1 M$ defined by $\Gamma^i \pi^1 = \Ima (\pi^1 M(1,i-1) \to \pi^1 M)$.
\begin{theorem} \emph{\cite[Theorem~12.8]{bousfield}}
Let $(M(X),d)$ be the minimal model of $X$. Then 
$$\Hom_{\Q}(\pi^1 M(X) / \Gamma^i \pi^1 M(X), \Q) \cong (N_i \pi_1 X) \otimes \Q.$$
\end{theorem}

There is a second construction of $\hat{G} \otimes \Q$ that proceeds through the theory of complete Hopf algebras \cite[Appendix~A]{quillen}. In particular, $\hat{G} \otimes \Q$ is isomorphic to the group of group-like elements $\{x \in \overline{\Q[G]}\ |\ \Delta x=x \hat{\otimes} x\}$ in the completion of the group algebra $\Q[G]$ with respect to the augmentation ideal. We can also define a complete Lie algebra by taking primitive elements $$L(G):=\{x\in \overline{\Q[G]}\ |\ \Delta x=1 \hat{\otimes} x+x\ \hat{\otimes} 1\}$$ in $\overline{\Q[G]}$. Recall that there is a lower central series for Lie algebras defined by $\Gamma_1 L = L$ and $\Gamma_{i+1} L = [L, \Gamma_i L]$.

The following proposition tells us that understanding the Malcev Lie algebra $L(G)$ is enough to understand the quotients in the lower central series of $G$ and also the completion $\overline{\Q[G]}$ with respect to the augmentation ideal. 
\begin{proposition} \emph{\cite[Section~4]{romanbez}}
\begin{enumerate}
\item $(\Gamma_i/\Gamma_{i+1} G) \otimes \Q \cong \Gamma_i/\Gamma_{i+1} (\hat{G} \otimes \Q) \cong \Gamma_i/\Gamma_{i+1} L(G)$. 
\item $\overline{U(L(G))} \cong \overline{\Q[G]}$ as complete Hopf algebras.
\end{enumerate}
\end{proposition}
In fact, even more is true. The power series defining $\log(x)$ and $e^x$ converge in any complete Hopf algebra and give an isomorphism between the Lie algebra of primitive elements and the group of group like elements \cite[Appendix~A.2]{quillen}. Thus, $L(G)$ completely determines $\hat{G} \otimes \Q$ and vice versa.

\subsection{Nonhomogeneous Quadratic Duality}
\label{dualitysubsection}

In this subsection, we describe and outline definitions and results on the nonhomogeneous quadratic duality \cite{priddy, positselskii, romanbez} between quadratic differential graded algebras and weak quadratic-linear algebras. This will give a tractable method for computing the minimal model $(M(X),d)$ and the Malcev Lie algebra $L(\pi_1(X))$ when we have a quadratic model $(A(X), d)$ for a space $X$. 

First we need to establish some conventions on graded and filtered algebras.
All graded and filtered algebras will be locally finite dimensional. 
All gradings will be concentrated in nonnegative degree and will be notated with superscripts. All $\N$-filtrations will be increasing, exhaustive, and indexed by subscripts. The tensor algebra on a $k$-vector space $V$ will be denoted by $T(V)$. It is graded by putting $V$ in degree $1$, equipped with the increasing filtration induced by the grading, and augmented by the map $\epsilon: T(V) \to k$ that sends $V$ to $0$. 

A \textbf{WQLA} (weak quadratic-linear algebra) is an augmented algebra $\epsilon: B \to k$ together with a choice of $k$-subspace $W$, satisfying:
\begin{enumerate}[(i)]
\item $1 \in W$.
\item $B$ is generated multiplicatively by $W$.
\item Let $V = \ker(\epsilon|_W)$ and $J=\ker(T(V)\to B)$.  The ideal $J$ is generated by $J_2$.
\end{enumerate}
A \textbf{QLA} (quadratic-linear algebra) is an augmented algebra $\epsilon: B \to k$ equipped with an exhaustive $\N$-filtration such that $\gr B$ is quadratic. In particular this implies that the choice $W = B_1$ makes $B$ into a $WQLA$. A morphism of WQLAs $f: (B, \epsilon, W) \to (B', \epsilon', W')$ is a homomorphism of augmented algebras such that $f(W) \subset W'$. Morphisms of QLAs coincide with homomorphisms of augmented filtered algebras.

A WQLA $B$ has an associated quadratic algebra $B^{(0)}$ which is defined by generators $V \cong W / k \cdot 1$ subject to the relations $I = J_2/J_1$. For QLAs, $\gr B \cong B^{(0)}$. We say that a WQLA $B$ is \textbf{Koszul} if the underlying quadratic algebra $B^{(0)}$ is Koszul. Every Koszul WQLA is in fact a QLA \cite[Section~3.3]{positselskii}.
Denote the category of weak quadratic-linear algebras  by $\mathtt{WQLA}$, and denote its subcategory consisting of Koszul quadratic-linear algebras by $\mathtt{KLA}$. 

A DGA $A$ has an underlying graded algebra given by forgetting the differential. We say that a DGA is \textbf{quadratic} (respectively \textbf{Koszul}) if its underlying algebra is quadratic (respectively Koszul). 
Let $\mathtt{QDGA}$ be the category of quadratic differential graded algebras, and denote its subcategory consisting of Koszul DGAs by $\mathtt{KDGA}$. 

There is a fully faithful contravariant functor $D_{WQLA}: \mathtt{WQLA} \to \mathtt{QDGA}$ that is defined on objects as follows: Let $(B, \epsilon, W)$ be a WQLA. As a graded algebra $D(B, \epsilon, W) = B^{(0)!}$ is the quadratic dual to the quadratic algebra associated to $B$. Note that $J_2 \cap T_1(V) = 0$ and $I = J_2/J_1$. Thus we can represent $J_2$ as the graph of a linear map $(-h, -\phi): I \to T_1(V) = k \oplus V$. It is easy to see that $J_2 \subset \ker(\epsilon)$ implies that $h = 0$. The map $d_1 = \phi^*: D(B, \epsilon, W)^1 \cong V^* \to I^* \cong D(B, \epsilon, W)^2$ can be extended to a differential on $D(B, \epsilon, W)$.

On the other hand, there is a contravariant functor $D_{QDGA}: \mathtt{QDGA} \to \mathtt{WQLA}$ defined on objects as follows: Let $(A, d)$ be a quadratic DGA and let $V=A^1$. We can write $A \cong T(V)/J$. The map $d|_{A^1}:V=A^1\to A^2 = (V\otimes V)/J$ has a dual map $\phi:J^\perp\to V^*$, where $V^*$ is the dual vector space to $V$ and $J^\perp\subseteq V^*\otimes V^*$ is the annihilator of $J$. Then $D(A, d) = (T(V^*)/I, \bar{\epsilon}, W_A)$, where $I$ is the ideal generated by $\{x-\phi(x)\ |\ x\in J^\perp\}$,  $\bar{\epsilon}$ is the augmentation induced by $\epsilon:T(V^*)\to k$, and $W_A=k\oplus V^*$.

\begin{proposition} \emph{\cite[Section~2.5]{positselskii}}
The functors $D_{WQLA}$ and $D_{QDGA}$ restrict to a contravariant equivalence of categories between $\mathtt{KLA}$ and $\mathtt{KDGA}$.
\begin{center}
\begin{tikzpicture}
\node at (0,0) (KLA) {$\mathtt{KLA}$};
\node at (2,0) (KDGA) {$\mathtt{KDGA}$};
\node at (0,1) (WQLA) {$\mathtt{WQLA}$};
\node at (2,1) (QDGA) {$\mathtt{QDGA}$};
\draw[->] (WQLA.350) -- (QDGA.190);
\draw[->] (QDGA.170) -- (WQLA.10);
\draw[<->] (KLA) -- (KDGA);
\node at (1,.1) {$\sim$};
\draw[right hook->] (KLA) -- (WQLA);
\draw[right hook->] (KDGA) -- (QDGA);
\end{tikzpicture}
\end{center}
\end{proposition}

If $(A,d)$ is a commutative QDGA, then $S^2(V)\subseteq J$ and hence $J^\perp\subseteq\Lambda^2(V^*)$. This implies that there is a Lie algebra $L=L(A)$ such that $$D(A, d) \cong (U(L), \epsilon, W_A);$$ we call this Lie algebra the \textbf{Lie algebra dual to $A$}. Note that in general, $W_A \neq k \oplus L$, so the induced filtration is not the order filtration.

\begin{example}
We start with a very simple example of the duality we consider in Section \ref{topologysection}. A model for the punctured elliptic curve is given by: 
$$A=\Lambda(x,y,g)/(xg,yg)$$ with differential defined by $dx=dy=0$ and $dg=xy$. The QL-algebra dual to $A$ is 
$$B = T(a,b,c)/(ab-ba-c).$$
As augmented algebras $B \cong U(L)$, where $L$ is the free Lie algebra on the generators $a$ and $b$, but the filtration on $B$ does not coincide with the order filtration on $U(L)$.
\end{example}

Let $L$ be a finite dimensional Lie algebra, and consider its universal enveloping algebra $U(L)$ equipped with the order filtration. It is an easy exercise to see that the QDGA dual to $U(L)$ is the graded algebra $\Lambda(L^*)$ equipped with differential dual to the Lie bracket. This is often called the the standard (or Chevalley-Eilenberg) complex of $L$ and is denoted by $(\Omega(L),d)$. We can extend this to the situation when the Lie algebra $L^\bullet$ is $\N$-graded with each graded piece finite-dimensional by defining the standard complex of $L^\bullet$ to be the restricted dual subalgebra $\Omega(L^\bullet):=\oplus_{i,j}([\Lambda^iL]^j)^*$ of $\Lambda(L^*)$ with differential dual to the Lie bracket.

Given a minimal model $(M(X), d)$ for a space $X$, we can reconstruct the Lie algebra $L(\pi_1X)$ as follows: The commutative DGAs $(M(1, i), d)$ are quadratic and hence dual to Lie algebras $L_i$. Moreover the inclusions $M(1, i) \to M(1, i+1)$ induce maps $L_{i+1} \to L_i$. 
\begin{theorem} \emph{\cite[Theorem~13.2]{griffithsmorgan}}
There are natural isomorphisms $$L_i \cong L(\pi_1 X)/\Gamma_i L(\pi_1X)\ \ \ \ \text{ and }\ \ \ \ L(\pi_1 X) \cong \varprojlim L_i.$$ 
\end{theorem}

We can also recover $\overline{\Q[\pi_1 X]}$ from $(M(X), d)$. Let $(C^{\bullet, \bullet}(A, d), d_1, d_2)$ be the cobar bicomplex of $(A, d)$ as defined in \cite[Section~3]{romanbez} (also called the dual bar bicomplex in \cite[Section~3]{positselskii}) and let $H^*_b(A, d)$ be the cohomology of its totalization.
\begin{lemma} \emph{\cite[Lemma~3.1]{romanbez}} \label{cobarprop}
Let $(A, d)$ be a QDGA. Then $H^0_b(A,d)$ is naturally isomorphic to $D(A, d)$. The increasing columns filtration on the cobar complex induces the QLA structure on $D(A, d)$. The decreasing rows filtration on the cobar complex induces the filtration by powers of the augmentation ideal of $D(A, d)$.
\end{lemma} 

\begin{proposition} \emph{\cite[Proposition~4.0]{romanbez}} \label{cobarprop2}
Let $(M(X),d)$ be the minimal model of $X$. 
Then $$\overline{\Q[\pi_1 X]} \cong \overline{H^0_b(M(X), d)}$$
where the completion on the left is with respect to the augmentation ideal and the completion on the right is with respect to the decreasing rows filtration.
\end{proposition}

Suppose that $(A,d)$ is a quadratic model for a space $X$. Let $L$ be the Lie algebra dual to $A$, and let $L(\pi_1X)$ be the Malcev Lie algebra of $X$.
The following theorem tells us how to obtain $L(\pi_1X)$ from $L$, and also how to compute the minimal model of $X$ when $A$ is Koszul. 
\begin{theorem}
\label{liealgebrathm}
Let $X$ be a space with a quadratic model $(A(X),d)$, and let $L=L(A(X))$ be the Lie algebra dual to $A(X)$. 
\begin{enumerate}
\item $\overline{U(L)}\cong\overline{\Q[\pi_1X]}$, where the completions are each with respect to the augmentation ideal. This isomorphism respects the Hopf algebra structures.
\item $\overline{L}\cong L(\pi_1X)$, where the completion of $L$ is with respect to the filtration by bracket length.
\item If $A(X)$ is Koszul, and $L$ is graded by bracket length with $L_i:=L/\Gamma_iL$ finite dimensional for all $i$, then $(\Omega(L^\bullet),d)$ is the minimal model of $X$.
\item Under the hypotheses of (3), $\Q_\infty(X)$ is a $K(\pi,1)$ space.
\end{enumerate}
\end{theorem}
\begin{proof}
Since $(A(X),d)$ is a model for $X$, there is a quasi-isomorphism from the minimal model $(M(X),d)$ to $(A(X),d)$.
This gives a map on the degree-zero cohomology of their cobar complexes, $H^0_b(C_{A(X)})\to H^0_b(C_{M(X)})$, which induces an isomorphism on the associated graded quotients with respect to the rows' filtration on the complexes \cite[Lemma~3.3a]{romanbez}.
Thus, there is an isomorphism on the completions with respect to the row filtration, $\overline{H^0_b(C_{A(X)})}\cong \overline{H^0_b(C_{M(X)})}$.
Also by Proposition \ref{cobarprop}, $H^0_b(C_{A(X)})$ is the dual to $A(X)$, $U(L)$, and the rows' filtration is the filtration by the augmentation ideal.
Hence $\overline{H^0_b(C_{A(X)})}\cong \overline{U(L)}$.
Moreover, Proposition \ref{cobarprop2} says that $\overline{H^0_b(C_{M(X)})}\cong \overline{\Q[\pi_1X]}$, completing the proof of (1). 

Since the isomorphism in (1) respects the Hopf algebra structures, taking the primitive elements on each side yields the isomorphism in (2).

The projection $f: (U(L), \epsilon, W_{A(X)})\to (U(L_i), \epsilon, \Q \oplus L_i)$ is a map of QLAs, and so it induces a map on the dual DGAs $g:(\Omega(L_i),d_i)\to (A(X),d)$. The grading of $L$ by bracket length induces another grading on $U(L)$ and on $U(L_i)$ which we call weight, and the map $f$ is an isomorphism for weight $j<i$. The weight gradings on $U(L)$ and $U(L_i)$ also induce weight gradings on $\Ext_{U(L)}^*(\Q,\Q)$, $\Ext_{U(L_i)}^*(\Q,\Q)$, $(A(X), d)$, and $(\Omega(L_i), d_i)$. The differentials on $(A(X), d)$ and $(\Omega(L_i), d_i)$ preserve weight and hence we have a weight grading on $H^*(A(X),d)$ and $H^*(\Omega(L_i),d_i)$.
Consider the following diagram for weight $j<i$:
\begin{center}
\begin{tikzpicture}[scale=2]
\node at (0,0) (HA) {$H_j^*(A,d)$};
\node at (2,0) (Hi) {$H_j^*(\Omega(L_i),d_i)$};
\node at (0,1) (EU) {$\Ext_{U(L),j}^*(\Q,\Q)$};
\node at (2,1) (Ei) {$\Ext_{U(L_i),j}^*(\Q,\Q)$};
\draw[->] (Hi) --node[above]{$g^*$} (HA);
\draw[->] (Ei) --node[above]{$f^*$} (EU);
\draw[<->] (HA) -- (EU);
\draw[<->] (Hi) -- (Ei);
\end{tikzpicture}
\end{center}
Since $A(X)$ and $U(L_i)$ are Koszul, the maps on the right and left are both isomorphisms \cite[Lemma~3.2]{romanbez}. 
Since $U(L)$ and $U(L_i)$ agree for weight $j<i$, one can see that $f^*$ is an isomorphism for weight $j<i$ by comparing the minimal graded free resolutions of $\Q$ considered as a $U(L)$-module and as a $U(L_i)$-module. Thus, the map $g$ is a quasi-isomorphism for weight $j<i$. 
Since $\Omega(L^\bullet)=\varprojlim_i\Omega(L_i)$, we have a quasi-isomorphism from the standard complex of $L$ to $A(X)$. 
Moreover, $\Omega(L^\bullet)$ is minimal and is hence the minimal model of $A(X)$. 

Finally, notice that the minimal model $\Omega(L^\bullet)$ is generated in degree 1. By Proposition \ref{ratlhomotopyprop}, this happens exactly when the space is rationally $K(\pi,1)$. 
\end{proof}

\section{Topology of $X_\A$}\label{topologysection}

In this section, we will first review known results on the rational homotopy theory of linear arrangements. These results will apply to the toric case as well, and so we focus on proving the analogous results for abelian arrangements. 
In the projective case (with curves of positive genus), we compute the quadratic dual to the QDGA $(A(X_\A),d)$ for a chordal arrangement $\A$ and give a combinatorial presentation for the Lie algebra dual to $A(X_\A)$. 
This then gives us a combinatorial description of $\overline{\Q[\pi_1 X_\A]}$, the Malcev Lie algebra $L(\pi_1 X_\A)$, and the minimal model $(M(X_\A),d)$. 
Finally, we will show that $X_\A$ is a rational $K(\pi,1)$ space. 

\subsection{Linear and Toric Arrangements}

Let $\Gamma=(\V,\E)$ be a chordal graph, and let $\A=\A(\Gamma,\C)$. 
Papadima and Yuzvinsky \cite{papadimayuz} describe the \textbf{holonomy Lie algebra}, $L$, of $X_\A$ and show that it is the Lie algebra dual to the cohomology ring $H^*(X_\A)$. 
They also show that the standard complex of $L$ is the minimal model of $X_\A$ \cite[Propositions~3.1\&4.4]{papadimayuz}.
Moreover, Kohno \cite{kohno} shows that the holonomy Lie algebra is isomorphic to the Malcev Lie algebra $L(\pi_1 X_\A)$. 

This Lie algebra $L$ can be described as the free Lie algebra generated by $c_e$ for $e\in\E$ modulo the relations
\begin{enumerate}[(i)]
\item $[c_e,c_{e'}]=0$ if $e$ and $e'$ are not part of a cycle of size 3, and 
\item $[c_{e_1},c_{e_2}+c_{e_3}]=0$ if $\{e_1,e_2,e_3\}$ is a cycle.
\end{enumerate}

If $X$ is a formal space, then $H^*(X)$ is Koszul if and only if $X$ is rationally $K(\pi,1)$ \cite[Theorem~5.1]{papadimayuz}. In particular, $X_\A$ is a rational $K(\pi,1)$ space.
Falk first showed that $X_\A$ is a rational $K(\pi,1)$ space when studying the minimal model \cite[Proposition~4.6]{falk}, but the generality of Papadima and Yuzvinsky's arguments allows us to directly apply them to toric arrangements, giving the following result:
\begin{theorem}
Let $\Gamma=(\V,\E)$ be a chordal graph and  $\A=\A(\Gamma,\C^\times)$. 
\begin{enumerate}
\item The holonomy Lie algebra of $X_\A$ is the Lie algebra dual to $H^*(X_\A)$, denoted by $L=L(H^*(X_\A))$. 
\item The minimal model of $X_\A$ is the standard complex of $L$, $(\Omega(L^\bullet),d)$. 
\item $X_\A$ is a rational $K(\pi,1)$ space.
\end{enumerate}
\end{theorem}
However, the presentation for the Lie algebra is much more complicated. 

\subsection{Abelian Arrangements and Higher Genus}

For this subsection, fix a projective curve $C$ of genus $g>0$ and a chordal graph $\Gamma=(\V,\E)$, and consider the chordal abelian arrangement $\A=\A(\Gamma,C)$. We will use quadratic-linear duality to study the rational homotopy theory of $X_\A$. 

Let $L$ be the free Lie algebra generated by $a_v^i,b_v^i,c_e$ for $v\in\V$,  $e\in\E$, and $i=1,\dots,g$ subject to the following relations:
\begin{enumerate}
\item[(i)] $[a_v^i,a_w^j]=[b_v^i,b_w^j]=0$ for $v,w\in \V$ with $v\neq w$,
\item[(iia)] $[b_{h(e)}^i,a_{t(e)}^i]=[b_{t(e)}^i,a_{h(e)}^i]=c_e$ for $e\in\E$,
\item[(iib)] $[a_v^i,b_w^j]=0$ if $v\neq w$ and there is no edge connecting $v$ and $w$, or if $i\neq j$,
\item[(iic)] $\displaystyle\sum_{i=1}^g[a_v^i,b_v^i]=\displaystyle\sum_{\substack{h(e)=v\\ \text{or }t(e)=v}}\!\!\!\!\! c_e$ for $v\in\V$,
\item[(iiia)] $[a_v^i,c_e]=[b_v^i,c_e]=0$ for $e\in \E$ and $h(e)\neq v\neq t(e)$, 
\item[(iiib)] $[a_{h(e)}^i+a_{t(e)}^i,c_e]=[b_{h(e)}^i+b_{t(e)}^i,c_e]=0$ for $e\in\E$,
\item[(iva)] 
$[c_{e},c_{e'}]=0$ whenever $e$ and $e'$ are not part of a 3-cycle, and
\item[(ivb)] $[c_{e_1},c_{e_2}+c_{e_3}]=0$ whenever $\{e_1,e_2,e_3\}$ is a cycle.
\end{enumerate}

The following theorem generalizes the main theorem of \cite{romanbez}.
Using the Lie algebra dual to $A(\A)$, this theorem gives a description of the Malcev Lie algebra of $X_\A$ when $\A$ is chordal. 

\begin{theorem}
\label{dualitythm}
Let $\Gamma=(\V,\E)$ be a chordal graph, $\A=\A(\Gamma,C)$, and $L$ be the Lie algebra described above. 
Then we have the following:
\begin{enumerate}
\item Consider the universal enveloping algebra $U(L)$ as a QLA whose first filtered piece is spanned by $a^i_v,b^i_v,c_e$ for $v\in \V$ and $e\in\E$. Then $U(L)$ is a Koszul QLA which is the nonhomogeneous quadratic dual to the Koszul DGA $A(\A)$.
\item $\overline{U(L)}\cong\overline{\Q[\pi_1(X_\A)]}$, where the completions are each with respect to the augmentation ideal. This isomorphism respects the Hopf algebra structures.
\item $\overline{L}\cong L(\pi_1(X_\A))$, where the completion of $L$ is with respect to the filtration by bracket length.
\end{enumerate}
\end{theorem}
\begin{proof}
(1) We can identify the QLA dual to $A(\A)$ with $L$ as follows. Let $a_v^i$, $b_v^i$, $c_e$ be the dual basis to $x_v^i$, $y_v^i$, $g_e$. The relations in the quadratic dual correspond to quadratic elements of the basis from Lemma \ref{basislemmahighergenus} since there is a natural isomorphism $\phi: I^\perp \cong ((V \otimes V)/I)^*$. The four types of relations (i)-(iv) in the presentation for $L$ come from four types of basis elements for $(V \otimes V)/I$ :
\begin{enumerate}
\item[(i)] $x_v^ix_w^j$ or $y_v^iy_w^j$ for $v\neq w$,
\item[(ii)] $x_v^iy_w^j$ for $v\neq w$ or $v=w$ and either $i\neq j$ or $i=j=1$,
\item[(iii)] $x_v^ig_e$ or $y_v^ig_e$ for $v\neq h(e)$,
\item[(iv)] $g_{e_1}g_{e_2}$ for $\{e_1,e_2\}$ not a broken circuit).
\end{enumerate}
The further subtypes in the relations for $L$ arise when computing $\phi^{-1}$. Since $A(\A)$ is Koszul by Theorem \ref{koszuldgathmabelian} or \ref{koszuldgathmhighergenus}, $U(L)$ is also Koszul.

Statements (2) and (3) follow from Theorem \ref{liealgebrathm}. 
\end{proof}

Since $A(\A)$ is a Koszul model for $X_\A$, Theorem \ref{liealgebrathm} gives us the following proposition and corollary, which describes the minimal model of $X_\A$ and shows that $X_\A$ is rationally $K(\pi,1)$. 

\begin{proposition}\label{minimalmodel}
Let $C$ be a complex projective curve of genus $g\geq1$, $\Gamma=(\V,\E)$  a chordal graph, $\A=\A(\Gamma,C)$, and $L$ be the Lie algebra described above. 
Consider $L^\bullet$ with the grading by bracket length. Then the standard complex $(\Omega(L^\bullet),d)$ is the minimal model for $X_\A$. 
\end{proposition}

\begin{corollary}\label{kpo}
Let $C$ be a complex projective curve of genus $g\geq1$, and let $\A=\A(\Gamma,C)$ be a chordal arrangement. Then its complement $X_\A$ is a rational $K(\pi,1)$ space.
\end{corollary}

\begin{remark}
Not only is $X_\A$ rationally $K(\pi,1)$, but it is not hard to show that $X_\A$ is also $K(\pi,1)$.
As an easy case, a punctured projective curve is homotopic to a wedge of circles and hence is $K(\pi,1)$.
Then by induction on $|\V|$ and using the long exact sequence in homotopy of a fibration, one can show if that if $\Gamma=(\V,\E)$ is chordal, then the complement to $\A(\Gamma-v,E)$ is $K(\pi,1)$.
The fibration arises as the restriction of the projection $C^\V\to C^{\V-v}$ to $X_{\A(\Gamma,C)}\to X_{\A(\Gamma-v,C)}$, where $v\in V$ is the maximum vertex in our perfect elimination ordering. The fiber of this fiber bundle is homeomorphic to $C\setminus\{k\text{ points}\}$ where $k=|\E\setminus(\E-v)|$.
\end{remark}

\begin{remark}
The fact that chordal arrangements are rationally $K(\pi,1)$ gives us a class of examples of abelian arrangements which are not formal. If we did have formality, then Theorem 5.1 of \cite{papadimayuz} would imply that the cohomology ring is Koszul. However, if the arrangement is chordal and has at least one cycle, the cohomology ring is not even generated in degree one and hence cannot be Koszul. 
\end{remark}

We end with an example computation of the first few terms $\Omega(L/\Gamma_iL)$ of the minimal model $\varprojlim_i\Omega(L/\Gamma_iL)$ for the complement of the elliptic braid arrangement of type $A_2$.

\begin{example}
Consider the case of an elliptic curve. The braid arrangement of type $A_2$ corresponds to the complete graph $\Gamma$ on three vertices  $\V=\{1<2<3\}$ with edges labeled $\{12,13, 23\}$. 

The DGA $A(\A)$ is the quotient of the exterior algebra $\Lambda(x_v,y_v,g_e)$ by the ideal generated by 
\begin{enumerate}[(i)]
\item $(x_i-x_j)g_{ij}$, $(y_i-y_j)g_{ij}$
\item $g_{12}g_{13}-g_{12}g_{23}+g_{13}g_{23}$
\end{enumerate}
with differential $dg_{ij}=(x_i-x_j)(y_i-y_j)$.

Recall from the proof of Theorem \ref{liealgebrathm} that the bracket-length defines another grading on $U(L)$, which also gives another grading on $A(\A)$ (by assigning the ``weight'' of a generator of $A(\A)$ to be the bracket length of its dual in $U(L)$). Notice that there is a quasi-isomorphism up to weight less than $i$ between $\Omega(L/\Gamma_iL)$ and $A(\A)$, for each $i$, which in the limit induces the quasi-isomorphism between $\Omega(L)$ and $A(\A)$. 

\begin{enumerate}
\item
$L/\Gamma_1L=0$ and hence $\Omega_1=\Q$, which is isomorphic to the weight-0 part of $A(\A)$. 
\item 
$L/\Gamma_2L$ is the vector space generated by $a_v$ and $b_v$ so that $\Omega_1=\Lambda(x_v,y_v)$ with differential $d_2=0$. This is isomorphic to the weight$\leq1$ part of $A(\A)$. 
\item
$\Omega(L/\Gamma_3L)=\Lambda(x_v,y_v,g_e)$ with differential $d_3:g_{ij}\mapsto(x_i-x_j)(y_i-y_j)$, which matches $A(\A)$ up to weight 2. 
\item $\Omega(L/\Gamma_4L)=\Lambda(x_v,y_v,g_e,k_{e,a},k_{e,b})$ where 
$$k_{e_a}:=[a_{h(e)},c_e]^*=-[a_{t(e)},c_e]^*$$ and $$k_{e,b}:=[b_{h(e)},c_e]^*=-[b_{t(e)},c_e]^*.$$ 
The differential $d_4$ restricts to $d_3$ on the subalgebra $\Omega(L/L^{(3)})$, and $d_4k_{e,a}=(x_{h(e)}-x_{t(e)})g_e$ and $d_4k_{e,b}=(y_{h(e)}-y_{t(e)})g_e$.
\item $\Omega(L/\Gamma_5L)=\Lambda(x_v,y_v,g_e,k_{e,a},k_{e,b},k_{e,aa},k_{e,bb},k_{e,ab},k_C)$ where 
$$k_{e,aa}:=[a_{h(e)},[a_{h(e)},c_e]]^*=-[a_{t(e)},[a_{h(e)},c_e]]^*,$$ 
$$k_{e,bb}:=[b_{h(e)},[b_{h(e)},c_e]]^*=-[b_{t(e)},[b_{h(e)},c_e]]^*,$$
$$k_{e,ab}:=[a_{h(e)},[b_{h(e)},c_e]]^*$$
and $$k_C=[c_{e_1},c_{e_2}]^*=-[c_{e_1},c_{e_3}]^*=[c_{e_2},c_{e_3}]^*$$
whenever $\{e_1,e_2,e_3\}$ is a cycle. The differential is defined by $$d_5k_{e,aa}=(x_{h(e)}-x_{t(e)})k_{e,a},$$ 
$$d_5k_{e,bb}=(y_{h(e)}-y_{t(e)})k_{e,b},$$ 
$$d_5k_{e,ab}=(x_{h(e)}-x_{t(e)})k_{e,b}+(y_{h(e)}-y_{t(e)})k_{e,a},$$ and $$d_5k_C=g_{e_1}g_{e_2}-g_{e_1}g_{e_3}+g_{e_2}g_{e_3}.$$
\end{enumerate}
\end{example}

\bibliographystyle{alpha}
\bibliography{qldrhtca}

\def\cprime{$'$}
\begin{thebibliography}{{Bib}13}

\bibitem[Bez94]{romanbez}
R.~Bezrukavnikov.
\newblock Koszul {DG}-algebras arising from configuration spaces.
\newblock {\em Geom. Funct. Anal.}, 4(2):119--135, 1994.

\bibitem[BG76]{bousfield}
A.~K. Bousfield and V.~K. A.~M. Gugenheim.
\newblock On {${\rm PL}$} de {R}ham theory and rational homotopy type.
\newblock {\em Mem. Amer. Math. Soc.}, 8(179):ix+94, 1976.

\bibitem[{Bib}13]{bibby}
C.~{Bibby}.
\newblock {Cohomology of abelian arrangements}.
\newblock {\em ArXiv e-prints}, arxiv: 1310.4866 [math.AG], October 2013.

\bibitem[BK72]{bousfieldkan}
A.~K. Bousfield and D.~M. Kan.
\newblock {\em Homotopy limits, completions and localizations}.
\newblock Lecture Notes in Mathematics, Vol. 304. Springer-Verlag, Berlin-New
  York, 1972.

\bibitem[Bri73]{brieskorn}
Egbert Brieskorn.
\newblock Sur les groupes de tresses [d'apr\`es {V}. {I}. {A}rnol\cprime d].
\newblock In {\em S\'eminaire {B}ourbaki, 24\`eme ann\'ee (1971/1972), {E}xp.
  {N}o. 401}, pages 21--44. Lecture Notes in Math., Vol. 317. Springer, Berlin,
  1973.

\bibitem[DCP05]{deconciniprocesi}
C.~De~Concini and C.~Procesi.
\newblock On the geometry of toric arrangements.
\newblock {\em Transform. Groups}, 10(3-4):387--422, 2005.

\bibitem[Dup]{dupont}
C.~Dupont.
\newblock {Hypersurface arrangements and a global Brieskorn-Orlik-Solomon
  theorem}.
\newblock To appear in {\it Annales de l'Institut Fourier}.

\bibitem[Fal88]{falk}
Michael Falk.
\newblock The minimal model of the complement of an arrangement of hyperplanes.
\newblock {\em Trans. Amer. Math. Soc.}, 309(2):543--556, 1988.

\bibitem[FG65]{fulkersongross}
D.~R. Fulkerson and O.~A. Gross.
\newblock Incidence matrices and interval graphs.
\newblock {\em Pacific J. Math.}, 15:835--855, 1965.

\bibitem[GM13]{griffithsmorgan}
Phillip Griffiths and John Morgan.
\newblock {\em Rational homotopy theory and differential forms}, volume~16 of
  {\em Progress in Mathematics}.
\newblock Springer, New York, second edition, 2013.

\bibitem[Koh83]{kohno}
Toshitake Kohno.
\newblock On the holonomy {L}ie algebra and the nilpotent completion of the
  fundamental group of the complement of hypersurfaces.
\newblock {\em Nagoya Math. J.}, 92:21--37, 1983.

\bibitem[OS80]{orliksolomon}
Peter Orlik and Louis Solomon.
\newblock Combinatorics and topology of complements of hyperplanes.
\newblock {\em Invent. Math.}, 56(2):167--189, 1980.

\bibitem[OT92]{orlikterao}
Peter Orlik and Hiroaki Terao.
\newblock {\em Arrangements of hyperplanes}, volume 300 of {\em Grundlehren der
  Mathematischen Wissenschaften [Fundamental Principles of Mathematical
  Sciences]}.
\newblock Springer-Verlag, Berlin, 1992.

\bibitem[Pos93]{positselskii}
L.~E. Positsel{\cprime}ski{\u\i}.
\newblock Nonhomogeneous quadratic duality and curvature.
\newblock {\em Funktsional. Anal. i Prilozhen.}, 27(3):57--66, 96, 1993.

\bibitem[Pri70]{priddy}
Stewart~B. Priddy.
\newblock Koszul resolutions.
\newblock {\em Trans. Amer. Math. Soc.}, 152:39--60, 1970.

\bibitem[PY99]{papadimayuz}
Stefan Papadima and Sergey Yuzvinsky.
\newblock On rational {$K[\pi,1]$} spaces and {K}oszul algebras.
\newblock {\em J. Pure Appl. Algebra}, 144(2):157--167, 1999.

\bibitem[Qui69]{quillen}
Daniel Quillen.
\newblock Rational homotopy theory.
\newblock {\em Ann. of Math. (2)}, 90:205--295, 1969.

\bibitem[Sul77]{sullivan}
Dennis Sullivan.
\newblock Infinitesimal computations in topology.
\newblock {\em Inst. Hautes \'Etudes Sci. Publ. Math.}, (47):269--331 (1978),
  1977.

\bibitem[SY97]{sheltonyuz}
Brad Shelton and Sergey Yuzvinsky.
\newblock Koszul algebras from graphs and hyperplane arrangements.
\newblock {\em J. London Math. Soc. (2)}, 56(3):477--490, 1997.

\bibitem[Yuz01]{yuz}
S.~Yuzvinski{\u\i}.
\newblock Orlik-{S}olomon algebras in algebra and topology.
\newblock {\em Uspekhi Mat. Nauk}, 56(2(338)):87--166, 2001.

\bibitem[Zie91]{ziegler}
G{\"u}nter~M. Ziegler.
\newblock Binary supersolvable matroids and modular constructions.
\newblock {\em Proc. Amer. Math. Soc.}, 113(3):817--829, 1991.

\end{thebibliography}

\end{document}